\documentclass{article}
\usepackage[a4paper, total={7in, 10in}]{geometry}
\usepackage{amsmath,amsthm,amssymb,mathrsfs,amstext, titlesec,enumitem, comment, graphicx, color, xcolor, 
stmaryrd,mathabx, lipsum,filecontents}
\usepackage{hyperref}
\usepackage{xpatch}
\hypersetup{colorlinks,linkcolor={magenta},citecolor={blue}}
\usepackage{xpatch}
\usepackage{tikz-cd}
\makeatletter
\def\Ddots{\mathinner{\mkern1mu\raise\p@
\vbox{\kern7\p@\hbox{.}}\mkern2mu
\raise4\p@\hbox{.}\mkern2mu\raise7\p@\hbox{.}\mkern1mu}}
\makeatother

\titleformat{\section}
{\normalfont\Large\bfseries}{\thesection}{1em}{}
\titleformat*{\subsection}{\Large\bfseries}
\titleformat*{\subsubsection}{\large\bfseries}
\titleformat*{\paragraph}{\large\bfseries}
\titleformat*{\subparagraph}{\large\bfseries}

\makeatother
\theoremstyle{Theorem}
\newtheorem{thm}{Theorem}[section]
\newtheorem{lem}[thm]{Lemma}
\newtheorem{qn}[thm]{Question}

\newtheorem{cor}[thm]{Corollary}

\theoremstyle{definition}
\newtheorem{defn}[thm]{Definition}

\newcommand{\N}{\mathbb{Z}^+}
\newcommand{\Z}{\mathbb{Z}}
\newcommand{\Q}{\mathbb{Q}}

\newcommand{\Po}{\mathbb{P}}

\newcommand{\p}{\mathcal{P}}
\newcommand{\bN}{\beta\mathbb{Z}^+}

\date{\vspace{-5ex}}

\begin{document}

\title{{\bf  Homogeneous Patterns in Ramsey Theory}}
%Nonlinear Extension of Monochromatic Pythagorean Triple and Rado's Conjecture Concerning Degree of Regularity of Equations}}

\author{
Sukumar Das Adhikari\footnote{Ramakrishna Mission Vivekananda Educational and Research Institute, Belur Math, Howrah, West Bengal-711202, India. \textit{adhikarisukumar@gmail.com}} \and
Sayan Goswami\footnotemark[1] \footnote{\textit{sayan92m@gmail.com}}
}

%\date{\vspace{-5ex}}

\makeatother

\maketitle

\begin{abstract}
In this article, we investigate homogeneous versions of certain nonlinear Ramsey-theoretic results, with three significant 
applications.

As the first application, we prove that for every finite coloring of $\mathbb{Z}^+$, there exist an infinite set $A$ and an 
arbitrarily large finite set $B$ such that $A \cup (A+B) \cup A \cdot B$ is monochromatic. This result resolves the finitary 
version of a question posed by Kra, Moreira, Richter, and Robertson regarding the partition regularity of $(A+B) \cup A \cdot B$ 
for infinite sets $A, B$ (see \cite[Question 8.4]{kra1}), which is closely related to a question of Erd\H{o}s (see \cite[Page 58]{erdos}).

As the second application, we make progress on a nonlinear extension of the partition regularity of Pythagorean triples. 
Specifically, we demonstrate that the equation $x^2 + y^2 = z^2 + P(u_1, \dots, u_n)$ is $2$-regular for certain appropriately 
chosen polynomials $P$ of any desired degree.

Finally, as the third application, we establish a nonlinear variant of Rado's conjecture concerning the degree of regularity. We 
prove that for every $m, n \in \mathbb{Z}^+$, there exists an $m$-degree homogeneous equation that is $n$-regular but not 
$(n+1)$-regular. The case $m = 1$ corresponds to Rado's conjecture, originally proven by Alexeev and Tsimerman \cite{AT}, 
and later independently by Golowich \cite{Gol}.
\end{abstract}
\noindent \textbf{Mathematics subject classification 2020:} 05D10, 05C55,  22A15, 54D35.\\
\noindent \textbf{Keywords:} Homogeneous patterns, Moreira theorem, Polynomial van der Waerden theorem, Monochromatic Pythagorean 
triples, Rado’s conjecture, Degree of regularity.

\section{Introduction}

Arithmetic Ramsey theory studies the existence of monochromatic patterns in any finite coloring of the integers or the set 
$\mathbb{Z}^+$ of positive integers. A ``coloring'' of a set refers to its division into disjoint subsets, and a subset is 
said to be ``monochromatic'' if it is entirely contained within one of these subsets.  

A collection $\mathcal{F}$ of subsets of $\mathbb{Z}^+$ is called \textbf{partition regular} if, for every finite coloring 
of $\mathbb{Z}^+$, there exists a monochromatic element $F \in \mathcal{F}$. Similarly, an equation $F(x_1, \dots, x_n) = 0$ 
over $\mathbb{Z}$ is \textbf{partition regular} if, for every finite partition of $\mathbb{Z}^+$, there exists a monochromatic 
solution of the equation $F$.  

The classification of partition regular equations remains a long-standing and highly challenging open problem in Ramsey theory.

One of the earliest results in Ramsey theory was established by Schur \cite{schur} in 1916, stating that the pattern 
$\{x, y, x+y : x \neq y\}$ is partition regular. By considering the mapping $n \mapsto 2^n$, one can immediately 
deduce that the multiplicative analogue of Schur’s theorem also holds: the pattern $\{x, y, x \cdot y : x \neq y\}$ is partition regular.  

A natural question arises: can both the additive and multiplicative versions of Schur’s theorem hold simultaneously for the same 
values of $x$ and $y$? This question has remained open for over a century, though it has appeared multiple times in the literature.

\begin{qn}\textup{(\cite[Question 11]{update}, \cite[Page 58]{erdos}, \cite[Question 3]{hls})}
    Is the pattern $\{x,y,x+y,xy\}$ partition regular?
\end{qn}

This question was investigated as early as 1979 by Hindman \cite{..} and Graham \cite{.} through brute-force computation, where they 
found affirmative answers for the case of a 2-coloring. More recently, in \cite{...}, Bowen provided the first mathematical proof of 
this result. However, for general finite colorings, the question remains open.  

Over finite fields, the problem was resolved by Green and Sanders (see \cite{san}). In \cite{Mor}, Moreira proved that for any finite 
coloring of $\mathbb{Z}^+$, there exist $x, y \in \mathbb{Z}^+$ such that $\{x, x+y, x \cdot y\}$ is monochromatic. Later, Alweiss 
provided a shorter proof of this result in \cite{al1}.

In \cite[page 58]{erdos}, Erdős posed the question of whether, for every finite coloring of $\N$, there exists an infinite 
sequence $a_1 < a_2 < \cdots < a_n < \cdots$ such that the set $\{a_i \cdot a_j, a_i + a_j : i < j, \text{ and } i, j \in \N\}$ 
is monochromatic. In 1984, Hindman \cite{conj} constructed a finite partition of $\mathbb{N}$ that does not contain this pattern.  

Recently, inspired by the result of Moreira (\cite{Mor}), Kra, Moreira, Richter, and Robertson posed a slight variation of Erdős' 
question in \cite[Question 8.4]{kra1}. They asked whether, for any finite coloring of $\mathbb{N}$, there exist two infinite sets 
$A$ and $B$ (not necessarily the same) such that $(A+B) \cup (A \cdot B)$ is monochromatic.  

This result was unknown even for $|A| = |B| = 2$ until Bowen's work in \cite{bowen}. In that article, Bowen proved that 
$A \cup (A+B) \cup (A \cdot B)$ is monochromatic  for some $A$, $B$,  where $A$ is infinite and $B$ is an arbitrarily large 
finite set. In this article, we present a very short proof of this result. This result serves as the first objective of this article.

\begin{thm}\label{task}
For any finite coloring of $\N,$ there exists an infinite set $A$ and finite set $B$ of arbitrary length such that 
$A\cup (A+B)\cup A\cdot B$ is monochromatic.
\end{thm}
\smallskip

The second goal of this article is to explore a nonlinear extension of the partition regularity of Pythagorean triples. While linear 
equations have relatively simple answers in terms of partition regularity, understanding this property for polynomial equations remains 
highly challenging. In particular, the following seemingly simple question posed by Erd\H{o}s and Graham remains open:

\begin{qn}\textup{\cite{pyth1, pyth2}}
    Is the equation \( x^2 + y^2 = z^2 \) partition regular?
\end{qn}

In 2016, with the help of a computer search \cite{comp}, M. Heule, O. Kullmann, and V. Marek confirmed this result for the case of 
2-colorings. The following theorem demonstrates that certain polynomials can be added to this equation while preserving 2-regularity.

Our second main result in this article is the following theorem:

\begin{thm}\label{xyz}
    For any \( n \in \N \) and \( c \in \mathbb{Q} \), the following equations are \(2\)-regular:
    \begin{enumerate}
        \item \label{xyz11} \( X^2 + Y^2 = Z^2 + c \cdot U^n V \),
        \item \label{xyz12} \( X^2 + Y^2 = Z^2 + c \cdot (U \pm V) \).
    \end{enumerate}
\end{thm}

We note that when \( c = 1 \), one case of (2) follows from the polynomial van der Waerden theorem, while the other follows from \cite{green} by simply setting \( X = Y = Z \). However, our proof ensures that \( X, Y, Z \) are mutually distinct.  

Another important observation is that our proof of Theorem \ref{xyz} implies that if \( X^2 + Y^2 = Z^2 \) is \( r \)-regular for some \( r > 2 \), then both equations in Theorem \ref{xyz} remain \( r \)-regular. For recent developments in this direction, see \cite{fr1, fr2}.
\smallskip

The third objective of this article is to establish a nonlinear extension of a conjecture by Rado concerning the degree of regularity. In 1927, van der Waerden \cite{vdw} proved that for any \( l, r \in \mathbb{Z}^+ \), there exists a positive integer \( W(l,r) \) such that, for any \( r \)-coloring of \( \{1,2, \ldots, W(l,r)\} \), there exists a monochromatic arithmetic progression of length \( l \).  

In 1933, Rado \cite{rado} extended this result by providing a necessary and sufficient condition for the partition regularity of a system of linear homogeneous equations. However, despite this foundational work, very few results are known for nonlinear equations.  

For any \( r, m \in \mathbb{Z}^+ \), an equation \( f(x_1,\dots ,x_m) = 0 \) over \( \mathbb{Z} \) is said to be \( r \)-regular if, for every \( r \)-coloring of \( \mathbb{Z}^+ \), there exists a monochromatic set \( \{x_1, \dots ,x_m\} \) such that \( f(x_1, \dots ,x_m) = 0 \). The smallest such \( r \) for which the equation is \( r \)-regular but not \( (r+1) \)-regular is called its \emph{degree of regularity}.  

In \cite{rado}, Rado conjectured that for every \( n \in \mathbb{Z}^+ \), there exists a linear homogeneous equation over \( \mathbb{Z} \) with degree of regularity equal to \( n \). In 2009, Alexeev and Tsimerman \cite{AT} confirmed this conjecture by proving that the equation  
\[
\left(1-\sum_{i=1}^{n-1}\frac{2^i}{2^i-1}\right)x_1+\sum_{i=1}^{n-1}\frac{2^i}{2^i-1}x_{i+1} = 0
\]
is \( (n-1) \)-regular but not \( n \)-regular.  

Subsequently, Golowich \cite{Gol} confirmed a conjecture of Fox and Radoi\'{c}i\v{c} \cite{rr} by proving that, for every \( n \geq 2 \), the equation  
\[
x_1+2x_2+\dots +2^{n-2}x_{n-1}-2^{n-1}x_n = 0
\]
is \( (n-1) \)-regular but not \( n \)-regular, providing an alternative proof of Rado’s conjecture.  

%\begin{thm}[Alexeev-Tsimerman \cite{AT} and Golowich \cite{Gol}]\label{ATG}
 %   For every \( n \in \mathbb{N} \), there exists a linear homogeneous equation that is \( n \)-regular but not \( (n+1) \)-regular.
%\end{thm}  

The third objective of this article is to establish a nonlinear extension of this result. We prove the following result.

\begin{thm}\label{newone}
    Let $p$ be a prime, and  $m,n\in \N$ be such that $m$ is not a zero divisor in $\Z_n.$ Then the equation 
$$M_n\equiv \sum_{i=1}^{n-1}p^{mi}X_{i,1} X_{i,2}^{m-1}-X_n^{m-1}X_{n+1}=0$$ is $n-1$ regular but not $n$ regular.

In other words, given $m,n\in \N$, there exists an $m$ degree equation which is $n$ regular but not $n-1$ regular.
\end{thm}

\section{Proof of Theorem \ref{task}}
To prove Theorem \ref{task}, we first establish the following result, which is known for integer polynomials as a consequence 
of \cite[Proof of Theorem 1.4]{Mor}. However, since our result involves rational polynomials, an independent proof is required. 
We adapt the argument from \cite{Mor} while incorporating the necessary modifications to extend the result to the rational setting.

Throughout our article $\Po$ will denote  the collection of all polynomials with rational coefficients having the constant term zero.
For any non-empty set $S$, $\p_f(S)$ will be the collection of all finite nonempty subsets of $S$.

\begin{thm}\label{essential}
    For any $r\in \N,$ $F\in \p_f(\Po)$ and for any $r$- coloring $\N=\bigcup_{i=1}^rC_i$ there exists $i\in [1,r]$ 
and $y\in \N$ such that  $$\left\lbrace x: \{ x,xy,x+f(y):f\in F\}\subset C_i\right\rbrace$$ is infinite.
\end{thm}

 To establish Theorem \ref{essential}, we require certain technical results concerning ultrafilters. Let $\beta \N$ denote 
the space of ultrafilters, and for any two $ p, q \in \beta \N $, define their sum via the relation  
\[
A \in p + q \iff \{x : -x + A \in q\} \in p,
\]
where $ -x + A = \{ y : x + y \in A \} $. It is well known that $(\beta \N, +)$ forms a compact right-topological semigroup in 
which the right action is continuous. The minimal two-sided ideal of this semigroup, denoted $ K(\beta \N, +) $, consists of minimal 
ultrafilters. A set belonging to any member of $ K(\beta \N, +) $ is termed a \textit{piecewise syndetic set}. For a comprehensive 
exposition on these concepts, we refer the reader to \cite{hs}.  

The following result, known as the polynomial van der Waerden theorem, plays a crucial role in our argument. For different proofs we 
refer the articles \cite{poly1,polynomial,poly3}.

\begin{thm}\textup{\cite[Theorem 3.6]{polynomial}}\label{ap1}  
For $F\in \p_f(\Po)$,
and any piecewise syndetic set $ A \subseteq \N $, there exists $ n \in \N $ such that  
\[
A \cap \bigcap_{P \in F} (-P(n) + A)
\]
is a piecewise syndetic set.  
\end{thm}  

With these preliminaries in place, we now proceed to the proof of Theorem \ref{essential}.  

\begin{proof}[Proof of Theorem \ref{essential}]  

We assume that the constant zero function $ \vec{0} : \N \to \{0\} $ belongs to $ F $. Consider an arbitrary finite coloring of $\N$ 
given by  
\[
\N = \bigcup_{i=1}^{r} C_i \cup \bigcup_{j=1}^{q} E_j,
\]
where each $ C_i $ is a piecewise syndetic set, while none of the $ E_j $ are piecewise syndetic.  

We inductively construct five sequences:  
\begin{itemize}  
    \item $ (t_n)_{n\geq 0} $ in $ \{1, \dots, r\} $,  
    \item $ (y_n)_{n\geq 1} $ in $ \N $, strictly increasing,  
    \item $ (B_n)_{n\geq 0} $, a sequence of piecewise syndetic subsets of $ \N $, satisfying $ B_n \subseteq C_{t_n} $,  
    \item $ (D_n)_{n\geq 1} $, a sequence of piecewise syndetic subsets of $ \N $, satisfying $ D_n \subseteq B_{n-1} $.  
\end{itemize}  

We begin by setting $ F_1 = F $, choosing $ t_0 = 1 $, and setting $ B_0 = C_1 = C_{t_0} $. By Theorem \ref{ap1}, we find 
$ y_1 \in \N $ such that  
\[
D_1 = B_0 \cap \bigcap_{P \in F_1} (B_0 - P(y_1))
\]
is a piecewise syndetic set. By \cite[Lemma 1.5]{update1}, the set $ y_1 D_1 $ is also piecewise syndetic. Choosing an ultrafilter 
$ p \in K(\beta \N, +) $ such that $ y_1 D_1 \in p $, we observe that $ \bigcup_{j=1}^{q} E_j \notin p $, and hence 
$ \bigcup_{i=1}^{r} C_i \in p $. Consequently, there exists $ t_1 \in \{1, \dots, r\} $ such that  
\[
B_1 = y_1 D_1 \cap C_{t_1}
\]
is piecewise syndetic.  

Assume that for some $ n \geq 1 $, we have constructed $ (t_m)_{m=0}^{n-1} $, $ (y_m)_{m=1}^{n-1} $, $ (B_m)_{m=0}^{n-1} $, and 
$ (D_m)_{m=1}^{n-1} $.  

For each $ z \in \N $ and $ P \in F $, define  
\[
P_z(n) = z P(n \cdot z) \quad \text{for all } n \in \N.
\]
Let  
\[
Y_n = \left\{ \prod_{t \in H} y_t : H \in \mathcal{P}_f(\{1, \dots, n-1\}) \right\}
\]
and define  
\[
F_n = \{ P_z \mid z \in Y_n \} \cup F.
\]
Choosing $ y_n \in \N $ such that  
\[
D_n = B_{n-1} \cap \bigcap_{P \in F_n} (B_{n-1} - P(y_n))
\]
is piecewise syndetic, we then pick $ t_n \in \{1, \dots, r\} $ such that  
\[
B_n = y_n D_n \cap C_{t_n}
\]
is piecewise syndetic.  

Observing that for each $ i \geq 1 $, we have $ B_i \subseteq y_i D_i \subseteq y_i B_{i-1} $, repeated iteration yields  
\[
\forall 0 \leq j < i, \quad B_i \subseteq y_{j+1} \cdots y_i B_j.
\]
Since the sequence $ (t_n) $ takes finitely many values, there exist indices $ j < i $ such that $ t_i = t_j $. Let $ \tilde{x} 
\in B_i \subseteq C_{t_i} $ and define $ y = y_{j+1} \cdots y_i $ and $ x = \tilde{x} / y $. Then,  
\begin{itemize}  
    \item $ xy = \tilde{x} \in B_i \subseteq C_{t_i} $,  
    \item $ \tilde{x} = xy \in B_i \subseteq y B_j $ implies $ x \in B_j \subseteq C_{t_j} = C_{t_i} $.  
\end{itemize}  

Finally,  
\[
y \cdot (x + P(y)) = \tilde{x} + y P(y) \in B_i + y P(y) \subseteq y_i D_i + y P(y),
\]
which further implies  
\[
x + P(y) \in B_j \subseteq C_{t_j} = C_{t_i},
\]
which completes the proof.  
\end{proof}

With Theorem \ref{ap1} established, we now proceed to the proof of Theorem \ref{task}. Our approach relies on leveraging 
Theorem \ref{ap1} in combination with several combinatorial arguments. The key idea is to utilize the structure of piecewise 
syndetic sets to extract a suitable configuration that satisfies the conditions of Theorem \ref{task}.

\begin{proof}[Proof of Theorem \ref{task}]
   Let $r,n\in \N$, and suppose that $\N$ is partitioned into $r$ colors by a coloring function $\omega$. Our objective is to 
identify a set $B$ with $|B| = n$ and an infinite set $A$ such that the set $A \cup (A + B) \cup (A \cdot B)$ is monochromatic 
under $\omega$.  

By a compactness argument (see \cite[Section 5.5]{hs}), there exists an integer $R \in \N$ such that, for any $r$-coloring of 
$[1, R]$, one can always find a monochromatic set of the form $\{a, ad, \dots, ad^n\}$. Let $\omega$ be the given $r$-coloring 
of $\N$. Define a new coloring $\omega'$ on $\N$ by  

\[
\omega' : \N \to \bigtimes_{i=1}^{R} \{1, \dots, r\}, \quad \text{where for } \alpha \in \N \text{ and } i \in [1,R], 
\quad \omega'(\alpha)_i = \omega(i\alpha).
\]

This construction ensures that  

\[
\omega'(\alpha) = \omega'(\beta) \quad \text{if and only if} \quad \omega(i\alpha) = \omega(i\beta) \text{ for all } i \in [1, R].
\]  

Next, define a finite set $F \subseteq \Q[x]$ by  

\[
F = \left\{ P(z) = \frac{m}{n} z \,:\, m, n \in [1, R] \right\}.
\]  

By Theorem \ref{essential}, there exist an integer $y \in \N$ and an infinite set $C \subseteq \N$ such that  

\[
\{x, x \cdot y, x + P(y) \mid P \in F, x \in C \}
\]  

is monochromatic under $\omega'$. Define the set  

\[
\mathcal{P} = \{ x, x \cdot y, x + P(y) \mid P \in F, x \in C \}.
\]  

Since $\omega'(x) = \omega'(x \cdot y) = \omega'(x + P(y))$ for all $P \in F$, this induces an $r$-coloring $\chi$ of $[1, R]$ given by  

\[
\chi(m) = \omega(m \cdot \mathcal{P}).
\]  

By the choice of $R$, there exist $a, d \in [1, R]$ such that  

\[
\omega\left(\{a, ad, \dots, ad^n\} \cdot \mathcal{P} \right) \text{ is constant}.
\]  

Define  

\[
D = \{ a, ad, \dots, ad^n \} \cdot \mathcal{P}.
\]  

Now, set  

\begin{itemize}
    \item $A = a \cdot C \subseteq D$,
    \item $B = \{ dy, d^2y, \dots, d^n y \}$.
\end{itemize}  

We claim that this choice of $A$ and $B$ satisfies the required conditions. Indeed, for any $ax \in A$ and $d^i y \in B$ with 
$1 \leq i \leq n$, we have  

\[
ax + d^i y = a \left( x + \frac{d^i}{a} y \right) \in D,
\]  

since $ad^i \leq R$ implies $a, d^i \leq R$. Moreover,  

\[
(ax) \cdot (d^i y) = (ad^i) \cdot (xy) \in D.
\]  

Thus, $A \cup (A + B) \cup (A \cdot B)$ is monochromatic, completing the proof. 

\end{proof}

\section{Proof of Theorem \ref{xyz}}

To establish a nonlinear extension of monochromatic Pythagorean triples, we introduce the concept of homogeneous sets, which 
provide a structured framework for analyzing partition regularity in a more constrained setting. The power of homogeneous sets 
lies in their ability to preserve combinatorial structure under finite colorings, making them a natural tool for extending 
classical results to nonlinear equations. Golowich \cite{Gol} first employed this notion to settle Rado's conjecture, demonstrating its 
effectiveness in tackling deep problems in Ramsey theory. We now formally define homogeneous sets and explore their role in our proof.

\begin{defn}  
A family \( S \) of subsets of \( \Z^+ \) is called \textit{homogeneous} if, for every set \( A \in S \) and every \( k \in \Z^+ \), 
the set  
\[ A_k = \{ka : a \in A\} \]  
also belongs to \( S \).  

Furthermore, for any \( r \in \N \), the family \( S \) is said to be \textit{\( r \)-regular} if, for every \( r \)-coloring 
of \( \N \), there exists a monochromatic set in \( S \).  
\end{defn}

For example, for any $r\in \N,$ the family of sets $\left\lbrace \{a,\frac{a}{2^i}\}:a\in \N,1\leq i\leq r 
\right\rbrace$ is a homogeneous family, which is $r$-regular. Again from  \cite{comp}, $\left\lbrace x,y,z\in 
\N:x^2+y^2=z^2\right\rbrace$ is a  homogeneous $2$-regular family.

\subsection{Proof of Theorem \ref{xyz}(1)}

To establish Theorem \ref{xyz}(1), we require a homogeneous extension of the polynomial van der Waerden theorem. This strengthened 
version ensures the existence of monochromatic polynomial patterns translated by homogeneous sets, allowing us to extend classical 
results to a broader nonlinear framework. The following theorem formalizes this idea.

\begin{thm}\label{pvdw}
    Let $\N$ be finitely colored and $F\in \p_f(\Po).$ Then there exists $a,d\in \N$ such that the pattern 
$$\{d\}\cup \{a+P(d):P\in F\}$$ is monochromatic. In addition, for any $m\in \N,$ we can choose $d$ such that $m|d.$
\end{thm}

Before proving Theorem \ref{pvdw}, we recall the notion of minimal idempotent ultrafilters in \((\bN, +)\). A minimal idempotent 
ultrafilter is an element of \( K(\bN, +) \) that satisfies the idempotency condition \( p = p + p \). The existence of such 
ultrafilters follows from Zorn’s Lemma, and any set belonging to a minimal idempotent ultrafilter is called a central set. 
Notably, if \( p \in K(\bN, +) \) is an idempotent ultrafilter, then for every \( n \in \N \), the set \( n\N \) belongs to \( p \).  

A key structural property of central sets is that they contain sequences with a rich combinatorial structure. Specifically, 
if \( A \) is a central set, there exists a sequence \( \langle x_n \rangle_n \) in \( \N \) such that its finite sums, given by  
\[
FS(\langle x_n \rangle_n) = \left\{ \sum_{t \in H} x_t : H \in \mathcal{P}_f(\N) \right\},
\]  
are entirely contained in \( A \). This property will play a crucial role in our proof of Theorem \ref{pvdw}.  

\begin{proof}[Proof of Theorem \ref{pvdw}]  
Let \( A \) be a central set. By definition, there exists a minimal idempotent ultrafilter \( p \in K(\bN, +) \) such 
that \( A \in p \). By \cite[Theorem 3.6]{polynomial}, the set  
\[
D = \left\{ n \in \N : A \cap \bigcap_{P \in F} (A - P(n)) \in p \right\}
\]  
belongs to \( p \). Since \( A \cap D \cap m\N \in p \), we can select a sequence \( \langle x_n \rangle_n \) in \( \N \) such 
that \( FS(\langle x_n \rangle_n) \subseteq A \cap D \cap m\N \). Choosing \( d = mn \in A \cap D \cap m\N \), we find \( a \in A \) 
such that  
\[
\{d\} \cup \{a + P(d) : P \in F\} \subseteq A.
\]  
This completes the proof.  
\end{proof}

Now we will prove the homogeneous version of the polynomial van der Waerden theorem. 

\begin{thm}[\textbf{Homogeneous Polynomial van der Waerden Theorem}]\label{rvdw}
    Let $S$ be a homogeneous family of subsets of $\N$ which is $r$-regular, and let $F\in \p_f(\Po)$. Then for any 
$q\in \N,$ and any $r$-coloring of $\N,$ there exists $B\in S,$ $d>0$ such that  $$\{qd\}\cup \{b,b+P(d):P\in F\}$$ 
is monochromatic. 
    In addition, for any $m\in \N,$ we can choose $d$ such that $m|d.$
\end{thm}

\begin{proof}
    By a compactness argument \cite[Section 5.5]{hs}, choose \( R \in \N \) such that for any \( r \)-coloring of \( [1,R] \), there 
exists a monochromatic set in \( S \). Let \( \omega \) be the given \( r \)-coloring of \( \N \).  

    We define a new \( r^R \)-coloring \( \omega' \) of \( \N \) as follows:  
    \[
    \omega' : \N \to \prod_{i=1}^{R} \{1, \dots, r\}, \quad \text{where for each } \alpha \in \N \text{ and } i \in \{1, \dots, R\}, 
\quad \omega'(\alpha)_i = \omega(i\alpha).
    \]  
    This construction ensures that  
    \[
    \omega'(\alpha) = \omega'(\beta) \text{ if and only if } \omega(i\alpha) = \omega(i\beta) \text{ for all } i \in [1,R].
    \]

    For each \( P \in \Po \) and \( r \in \Q \), define a new polynomial \( P_r \in \Po \) by  
    \[
    P_r(x) = P(rx).
    \]  
    Now, define the finite set of polynomials  
    \[
    F_1 = \left\{ \frac{1}{y} P_{\frac{z}{q}} : P \in F, \text{ and } y, z \in [1,R] \right\} \in \p_f(\Po).
    \]

    Applying Lemma \ref{pvdw} to the polynomial set \( F_1 \) and the coloring \( \omega' \), we obtain a monochromatic polynomial 
progression  
    \[
    \p = \{d_1\} \cup \{a + P(d_1) : P \in F_1\}.
    \]  

    By the definition of \( \omega' \), for each \( i \in [1,R] \), we have  
    \[
    \omega(x) = \omega(y) \quad \text{for all } x, y \in i \cdot \p = \{ i \cdot m : m \in \p \}.
    \]  

    Thus, defining the induced coloring \( \chi \) on \( [1,R] \) by  
    \[
    \chi(m) = \omega(m \cdot \p), \quad \text{for each } m \in [1,R],
    \]  
    we obtain a homogeneous set \( \{b_1, \dots, b_n\} \) with  
    \[
    \chi(b_1) = \chi(b_2) = \dots = \chi(b_n).
    \]  

    Now, we verify the desired properties:  
    \begin{enumerate}
        \item For every \( i,j \in [1,n] \), we have  
        \[
        \omega(b_i a) = \omega(b_j a).
        \]  
        \item Since \( d_1 \in \p \), we can choose \( d_1 \) such that \( q \mid d_1 \), ensuring that  
        \[
        b_1 d_1 = q \cdot \frac{b_1 d_1}{q} = qd \in b_1 \cdot \p, \quad \text{where } d = \frac{b_1 d_1}{q}.
        \]  
        \item For every \( i \in [1,n] \) and \( P \in F \), we have  
        \[
        \omega(ab_i + P(d)) = \omega\left( b_i \cdot \left( a + \frac{1}{b_i} P_{\frac{b_1}{q}}(d_1) \right) \right) = \chi(b_i).
        \]  
    \end{enumerate}

    Finally, defining \( B = \{ ab_i : 1 \leq i \leq n \} \), we conclude the proof.  
\end{proof}

Now we are in the position to prove Theorem \ref{xyz}(1). 

\begin{proof}[Proof of Theorem \ref{xyz}(1)]\label{wow1}

Let \( F \in \mathcal{P}_f(\Po) \) be a set that we specify. Our goal is to show that for every \( 2 \)-coloring of \( \mathbb{N} \), 
there exists \( (x,y,z) \in S \) such that the set  
\[
\{u\} \cup \{a + P(u) : a \in \{x, y, z\}, P \in F\}
\]  
is monochromatic.  

Define \( F = \{P_1, P_2, P_3\} \) where, for every \( z \in \mathbb{N} \), the polynomials are given by  
\[
P_1(z) = \frac{c z^n}{2}, \quad P_2(z) = \frac{c z^n}{4}, \quad P_3(z) = 0.
\]  
Then, the set  
\[
\left\{u, x + \frac{c u^n}{2}, x + \frac{c u^n}{4}, y, z \right\}
\]  
is monochromatic.  

Now, consider the equation:  
\begin{align*}
    \left( x + \frac{c u^n}{2} \right)^2 + y^2 &= z^2 + c u^n \left( x + \frac{c u^n}{4} \right) \\  
    &= z^2 + c u^n v.
\end{align*}  

Finally, define the variables as follows:  
\begin{itemize}
    \item \( X = x + \frac{c u^n}{2}, \quad Y = y, \quad Z = z; \)
    \item \( U = u, \quad V = x + \frac{c u^n}{4}. \)
\end{itemize}  
This completes the proof.
\end{proof}

\subsection{Proof of Theorem \ref{xyz}(2)}

To establish Theorem \ref{xyz}(1), we require a key technical lemma. This result extends Theorem 
\ref{essential} within the framework of homogeneous settings. The strengthened version guarantees the existence of the 
conclusion of Theorem \ref{essential}, 
but now applied to suitably translated homogeneous sets, thereby broadening its applicability.
\begin{lem}\label{anal}
     Let $S$ be a homogeneous family of subsets of $\N$ which is $r$-regular, and let $F\in \p_f(\Po)$. Then for any 
$r$-coloring of $\N,$ there exists $B\in S,$ $d>0$ such that $$\{b,b+P(d),b\cdot d:b\in B, P\in F\}$$ is partition regular.
\end{lem}
\begin{proof}
Let $R\in \N$ be the natural number as in the proof of Theorem  \ref{rvdw}. Let us define a new set of polynomials 
$F'\in \mathbb{P}$ such that $F'=\left\lbrace \frac{1}{n}P:P\in F, n\in [1,R]  \right\rbrace.$ Now
 for any finite coloring of $\N$, from Theorem \ref{essential} there exists $x,y\in \N$  such 
that $\{x,x\cdot y, x+P'(y):P'\in F'\}$ is monochromatic.
    Now the rest of the proof is similar to the proof of Theorem \ref{rvdw}, so we omit it.
\end{proof}

\begin{proof}[Proof of Theorem \ref{xyz}(2)]
Proceeding analogously to the proof of Theorem \ref{wow1}, an application of Lemma \ref{anal} ensures the existence of 
elements \( (x, y, z) \in S \) and \( d \in \N \) such that the set  
\[
\left\lbrace a + i \cdot \frac{c}{2} d, \, a \cdot d \, : \, i = 0, \pm1, \, a \in \{x, y, z\} \right\rbrace
\]  
is monochromatic, where \( x^2 + y^2 = z^2 \).

Furthermore, we observe that  
\[
\left(z + \frac{c}{2} d \right)^2 - \left(y \pm \frac{c}{2} d \right)^2 = x^2 + c \cdot (zd \mp yd).
\]  
A suitable redefinition of variables now completes the proof.

\end{proof}

\section{Proof of Theorem \ref{newone}}

Before proceeding with the proof of Theorem \ref{newone}, we recall some fundamental concepts from \( p \)-adic valuation theory, 
which was employed by Fox and Radoi\'{c}i\v{c}. These concepts play a crucial role in constructing counterexamples that demonstrate the 
non-regularity of certain equations.

Let \( p \) be a prime and \( r \in \N \). The \( p \)-adic valuation of \( r \), denoted by \( \operatorname{ord}_p(r) \), is 
defined as the unique integer \( m \) such that  
\[
r = p^m \cdot s, \quad \text{where } p \nmid s.
\]  
By convention, we set \( \operatorname{ord}_p(0) = \infty \). The following key properties hold for all \( x, y \in \N \):

\begin{enumerate}
    \item \textbf{Multiplicative Property:}  
        \[
        \operatorname{ord}_p(xy) = \operatorname{ord}_p(x) + \operatorname{ord}_p(y).
        \]  
   
    \item \textbf{Additive Property:}  
        \[
        \operatorname{ord}_p(x+y) \geq \min \{\operatorname{ord}_p(x), \operatorname{ord}_p(y)\},
        \]  
        with equality if and only if \( \operatorname{ord}_p(x) \neq \operatorname{ord}_p(y) \).\footnote{This follows from basic 
modular arithmetic.}
\end{enumerate}

While the \( p \)-adic valuation is sufficient for our discussion over natural numbers, it can be extended to the field of rationals. 
For further details, we refer to \cite{foxkl, padic}.

Before proving Theorem \ref{newone}, we present an application of Theorem \ref{rvdw}, which establishes the existence of certain 
quadratic equations whose degree of regularity lies between \( n-1 \) and \( n \) when \( n \) is even.

\begin{cor}\label{cor1}
    For every $a, n, c \in \Z^+$, the equation  
    \[
    \sum_{i=1}^{n-1}a^{2i}X_i^2 = X_n^2 + c \cdot X_{n+1}X_{n+2}
    \]
    is $(n-1)$-regular.

    In particular, setting $a = 2$ and $c = 2^{2n}$, if $n+1$ is odd, then the equation  
    \[
    L_n \equiv \sum_{i=1}^{n-1}2^{2i}X_i^2 - X_n^2 - 2^{2n} \cdot X_{n+1}X_{n+2} = 0
    \]
    is not $(n+1)$-regular. Consequently, for even $n \in \N$, we obtain  
    \[
    n-1 \leq dor(L_n) \leq n.
    \]
\end{cor}

\begin{proof}
    Let $M \in \N$ be such that $a^{n-1} \mid M$. Consider an $(n-1)$-coloring of the set  
    \[
    \left\lbrace M, \frac{M}{a}, \frac{M}{a^2}, \dots, \frac{M}{a^{n-1}} \right\rbrace.
    \]
    By the pigeonhole principle, there exist indices $0 \leq i < j \leq n-1$ such that  
    \[
    \frac{M}{a^i}, \frac{M}{a^j}
    \]
    are monochromatic. Consequently, for every $(n-1)$-coloring of $\N$, there exists a monochromatic element of the form 
$b, \frac{b}{a^i}$ for some $1 \leq i \leq n-1$. This ensures the existence of a homogeneous collection of such patterns.

    Next, let $F \in \p_f(\Po)$ be a finite set of polynomials defined as follows:
    \begin{enumerate}
        \item Define  
        \[
        M = \left\lbrace \frac{4}{c} \left( \sum_{i=1}^{n} a^{2i} - a^{2j} + \frac{c^2}{4} \right): j \in \{1, 2, \dots, n-1\} 
\right\rbrace.
        \]
        \item Let  
        \[
        F = \left\lbrace P(x) = \frac{Q}{4}x: Q \in M \right\rbrace.
        \]
    \end{enumerate}
    By applying Theorem \ref{rvdw}, for every $r$-coloring of $\N$, there exist $b \in \N$ and $j \in [1, n-1]$ such that  
    \[
    \{d\} \cup \{b + P(d) : P \in F\} \cup \left\lbrace \frac{b}{a^j} + P(d) : P \in F \right\rbrace
    \]
    is monochromatic. Now, we choose:
    \begin{itemize}
        \item For every $i \neq j \in [1, n-1]$, set $X_i = d$.
        \item Set  
        \[
        X_j = \frac{b}{a^j} + \frac{c}{2a^j}d.
        \]
        \item Define  
        \[
        X_n = b, \quad X_{n+1} = d, \quad X_{n+2} = b + \frac{Q}{4}d,
        \]
        where  
        \[
        Q = \frac{4}{c} \left( \sum_{i=1}^{n} a^{2i} - a^{2j} + \frac{c^2}{4} \right).
        \]
    \end{itemize}
    This provides a monochromatic solution to the equation  
    \[
    \sum_{i=1}^{n-1}a^{2i}X_i^2 = X_n^2 + c \cdot X_{n+1}X_{n+2}.
    \]

    To establish the second claim, define a coloring function $\chi: \N \to \Z_{n+1}$ by setting  
    \[
    \chi(r) = \operatorname{ord}_2(r) \mod (n+1) \quad \text{for all } r \in \N.
    \]
    Suppose, for contradiction, that this coloring yields a monochromatic solution, i.e., there exists $m \in \Z_{n+1}$ such that  
    \[
    \chi(X_i) = m \quad \text{for all } i \in [1, n+2].
    \]
    Then, we have:
    \begin{itemize}
        \item For every $1 \leq i \leq n-1$,  
        \[
        \operatorname{ord}_2(2^{2i} X_i^2) = 2i + 2m.
        \]
        \item For the square term $X_n^2$,  
        \[
        \operatorname{ord}_2(-X_n^2) = 2m.
        \]
        \item For the product term,  
        \[
        \operatorname{ord}_2(-2^{2n} X_{n+1} X_{n+2}) = 2n + 2m.
        \]
    \end{itemize}
    Note that  
    \[
    2i + 2m = 2j + 2m \iff 2(i - j) \equiv 0 \mod (n+1).
    \]
    However, if $n+1$ is odd,\footnote{If $n+1$ were even, we would have $2(i - j) = c(n+1)$ for some $c \geq 2$, implying $i - j = \frac{c}{2}(n+1) \geq n+1$, which is a contradiction.} this congruence is impossible. Consequently, each term in $L_n$ has a distinct $2$-adic valuation, leading to a contradiction since the sum of distinct $2$-adic valuations cannot be zero. Thus, a monochromatic solution does not exist for this coloring, which implies  
    \[
    dor(L_n) < n+1.
    \]
\end{proof}

In the preceding corollary, we examined an example of a nonlinear equation for which the exact degree of regularity remains unknown.  Now we  prove Theorem \ref{newone}.

\begin{proof}[Proof of Theorem \ref{newone}]
    Following a similar approach as in the proof of Corollary \ref{cor1}, for any $(n-1)$-coloring of $\mathbb{N}$, there exists a 
monochromatic pair of the form $\left( b, \frac{b}{p^{mj}} \right)$ for some $0< j \leq n-1$. 

    Consider the finite set of polynomials $F \subseteq \mathcal{P}_f(\mathcal{P})$ defined as follows:
    \begin{enumerate}
        \item Define the set 
        \[
        M = \left\lbrace \sum_{i=1}^{n-1} p^{mi} - p^{mj} + 1 : j \in \{1,2,\ldots,n-1\} \right\rbrace.
        \]
        \item Let 
        \[
        F = \left\lbrace P(x) = Qx : Q \in M \right\rbrace.
        \]
    \end{enumerate}
    
    By Theorem \ref{rvdw}, for every $(n-1)$-coloring of $\mathbb{N}$, there exist elements $b \in \mathbb{N}$ and $j 
\in \{1,2, \dots, n-1\}$ such that the set 
    \[
    \left\lbrace d \right\rbrace \cup \left\lbrace b + P(d) : P \in F \right\rbrace \cup \left\lbrace 
\frac{b}{p^{mj}} + P(d) : P \in F \right\rbrace
    \]
    is monochromatic. Now, define the elements:
    
    \begin{itemize}
        \item For every $i \neq j \in [1,n-1]$, set $X_{i,1} = X_{i,2} = d$.
        \item Set $X_{j,1} = \frac{b}{p^{mj}} + \frac{1}{p^{mj}} d$.
        \item Set $X_{j,2} = d$.
        \item Define $X_n = d$ and $X_{n+1} = b + Qd$, where $Q \in \sum_{i=1}^{n-1} p^{mi} - p^{mj} + 1$.
    \end{itemize}
    
    This provides a monochromatic solution to the equation $M_n = 0$. 

    To establish that $M_n$ is not $n$-regular, consider the coloring $\chi : \mathbb{N} \to \mathbb{Z}_n$ defined by 
    \[
    \chi(r) = \operatorname{ord}_p(r) \mod n, \quad \text{for all } r \in \mathbb{N}.
    \]
    Suppose, for contradiction, that this coloring yields a monochromatic solution. Then, for every $i \in [1, n+2]$, 
we must have $\chi(X_i) = r \leq n-1$. However, observing the order of $p$, we obtain:
    
    \begin{itemize}
        \item For every $1 \leq i \leq n-1$, 
        \[
        \operatorname{ord}_p \left( p^{mi} X_{i,1} X_{i,2}^{m-1} \right) = mi + rm.
        \]
        \item Similarly, 
        \[
        \operatorname{ord}_p \left( X_n^{m-1} X_{n+1} \right) = rm.
        \]
    \end{itemize}

    Now, we analyze all possible contradictions:
    
    \begin{itemize}
        \item If $i \neq j \in [1,n-1]$, then $rm + mi \neq rm + mj$, since $m$ has no zero divisors.
        \item If $i \in [1,n-1]$, then by a similar argument, $rm + mi \neq rm$; otherwise, we would have $mi \equiv 0 \pmod{n}$, 
which is impossible.
    \end{itemize}

    Since all elements of $M_n$ have distinct $p$-adic valuations and their sum equals zero, we reach a contradiction. 

    This completes the proof.
\end{proof}

\section*{Acknowledgement} The second author of this paper is supported by NBHM postdoctoral fellowship with 
reference no: 0204/27/(27)/2023/R \& D-II/11927.


\begin{thebibliography}{10}

\bibitem{al1} R. Alweiss: Monochromatic sums and products of polynomials, \emph{Discrete Anal.} \textbf{2024}:5, 7 pp.

%\bibitem{al2} R. Alweiss: Monochromatic sums and products over $\mathbb{Q}$. Preprint (2023), arXiv:2307.08901.

\bibitem{AT} B. Alexeev and J. Tsimerman: Equations resolving a conjecture of Rado on partition regularity, 
\emph{J. Combin. Theory Ser. A} \textbf{117} (2010), 1008–1010.

\bibitem{update} V. Bergelson: Ergodic Ramsey theory—an update, in \emph{Ergodic Theory of $\mathbb{Z}^d$ Actions} (Warwick, 1993–1994), 
London Math. Soc. Lecture Note Ser. \textbf{228}, Cambridge Univ. Press, Cambridge, 1996, pp. 1–61.

\bibitem{poly1} V. Bergelson and A. Leibman: Polynomial extensions of van der Waerden and Szemerédi theorems, 
\emph{J. Amer. Math. Soc.} \textbf{9} (1996), 725–753.

\bibitem{bjm} V. Bergelson, J. H. Johnson Jr., and J. Moreira: New polynomial and multidimensional extensions of classical partition results, 
\emph{J. Combin. Theory Ser. A} \textbf{147} (2017), 119–154.

\bibitem{...} M. Bowen: Monochromatic products and sums in 2-colorings of $\mathbb{N}$, \emph{Adv. Math.} \textbf{462} (2025), 110095.

%\bibitem{rational} M. Bowen and M. Sabok: Monochromatic products and sums in the rationals, \emph{Forum Math. Pi} \textbf{12} (2024), e17, 1–12.

\bibitem{bowen} M. Bowen: Composite Ramsey theorems via trees, arXiv:2210.14311v2.

\bibitem{erdos} P. Erd\H{o}s: Problems and results on combinatorial number theory. III, in \emph{Number Theory Day} (Proc. Conf., Rockefeller Univ., New York, 1976), 
Lecture Notes in Math. \textbf{626}, Springer, Berlin, 1977, pp. 43–72.

\bibitem{fr1} N. Frantzikinakis, O. Klurman, and J. Moreira: Partition regularity of Pythagorean pairs, 
To appear in \emph{Forum Math. Pi}, \textbf{13}  (2025) , e5.

\bibitem{fr2} N. Frantzikinakis, O. Klurman, and J. Moreira: Partition regularity of generalized Pythagorean pairs, arXiv:2407.08360.

\bibitem{rr} J. Fox and R. Radoi\'{c}i\v{c}: The axiom of choice and the degree of regularity of equations over the reals, 
Preprint (2005).

\bibitem{foxkl} J. Fox and D. J. Kleitman: On Rado's Boundedness Conjecture, 
\emph{J. Combin. Theory Ser. A} \textbf{113} (2006), 84–100.

\bibitem{Gol} N. Golowich: Resolving a conjecture on degree of regularity of linear homogeneous equations, 
\emph{Electron. J. Combin.} \textbf{21}(3) (2014).

\bibitem{padic} F. Q. Gouvêa: \emph{$p$-adic Numbers: An Introduction}, 2nd ed., Springer, Berlin, 1993.

%\bibitem{gos} S. Goswami: Monochromatic Translated Product and Answering Sahasrabudhe's Conjecture, submit/6086932.

\bibitem{pyth1} R. Graham: Some of my favorite problems in Ramsey theory, in \emph{Combinatorial Number Theory} (2007), 229–236, de Gruyter, Berlin.

\bibitem{pyth2} R. Graham: Old and new problems in Ramsey theory, in \emph{Horizons of Combinatorics}, 
Bolyai Soc. Math. Stud. \textbf{17} (2008), 105–118, Springer, Berlin.

\bibitem{.} R. L. Graham, B. L. Rothschild, and J. H. Spencer: \emph{Ramsey Theory}, 2nd ed., Wiley-Intersci. Ser. Discrete Math. Optim., 
John Wiley \& Sons, New York, 1990.

\bibitem{green} B. Green and A. Lindqvist: Monochromatic solutions to $x + y = z^2$, \emph{Canad. J. Math.} \textbf{71} (2019), 579–605.

\bibitem{san} B. Green and T. Sanders: Monochromatic sums and products, \emph{Discrete Anal.} (2016), Paper No. 613, 43 pp.

\bibitem{comp} M. Heule, O. Kullmann, and V. Marek: Solving and verifying the Boolean Pythagorean triples problem via cube-and-conquer, 
in \emph{Theory and Applications of Satisfiability Testing – SAT 2016}, Springer (2016), 228–245.

%\bibitem{fsum} N. Hindman: Finite sums from sequences within cells of partitions of $\mathbb{N}$, 
%J. Combin. Theory Ser. A \textbf{17} (1974), 1–11.

\bibitem{..} N. Hindman: Partitions and sums and products of integers, \emph{Trans. Amer. Math. Soc.} \textbf{247} (1979), 227–245.

\bibitem{conj} N. Hindman: Partitions and pairwise sums and products, \emph{J. Combin. Theory Ser. A} \textbf{37} (1984), 46–60.

\bibitem{hls} N. Hindman, I. Leader, and D. Strauss: Open problems in partition regularity, 
\emph{Combin. Probab. Comput.} \textbf{12} (2003), 571–583.

\bibitem{polynomial} N. Hindman: Problems and new results in the algebra of Beta S and Ramsey theory, 
in \emph{Unsolved Problems in Mathematics for the 21st Century}, J. Abe and S. Tanaka (eds.), IOS Press, Amsterdam (2001), 295–305.

\bibitem{hs} N. Hindman and D. Strauss: \emph{Algebra in the Stone-\v{C}ech Compactification: Theory and Applications}, 
2nd ed., de Gruyter, Berlin, 2012.

\bibitem{update1} N. Hindman and D. Strauss: Algebra in the Stone-\v{C}ech compactification—an update, 
\emph{Topology Proc.} \textbf{69} (2024), 1–69.

\bibitem{kra1} B. Kra, J. Moreira, F. K. Richter, and D. Robertson: Infinite sumsets in sets with positive density, 
\emph{J. Amer. Math. Soc.} \textbf{37} (2024), no. 3, 637–682.

\bibitem{kra2} B. Kra, J. Moreira, F. K. Richter, and D. Robertson: Problems on infinite sumset configurations in the integers and beyond, 
arXiv:2311.06197.

\bibitem{Mor} J. Moreira: Monochromatic sums and products in $\mathbb{N}$, 
\emph{Ann. of Math. (2)} \textbf{185} (2017), no. 3, 1069–1090.

\bibitem{rado} R. Rado: Studien zur Kombinatorik, \emph{Math. Z.} \textbf{36} (1933), 242–280.

\bibitem{schur} I. Schur: Über die Kongruenz $x^m + y^m \equiv z^m \,(\mathrm{mod}\, p)$, 
\emph{Jahresber. Dtsch. Math.-Ver.} \textbf{25} (1916), 114–117.

\bibitem{vdw} B. L. van der Waerden: Beweis einer baudetschen Vermutung, 
\emph{Nieuw Arch. Wisk.} \textbf{15} (1927), 212–216.

\bibitem{poly3} M. Walter: Combinatorial proofs of the polynomial van der Waerden theorem and the polynomial Hales–Jewett theorem, 
\emph{J. London Math. Soc.} \textbf{61}(1) (2000), 1–12.

\end{thebibliography}
\end{document}